\documentclass[12pt]{article}
\usepackage{mathrsfs}
\usepackage[active]{srcltx}
\usepackage[curve]{xypic}
\usepackage{graphics}
\usepackage{amsthm,amsfonts,amscd,amsmath,amssymb,color,array,verbatim,graphicx,caption2}
\input epsf.tex

\newtheorem{theorem}{Theorem}[section]
\newtheorem{lemma}[theorem]{Lemma}

\newtheorem*{theoremA}{Theorem A}

\def\HH{\mbox{$\mathbb H$}}
\def\CC{\mbox{$\mathbb C$}}

\def\ZZ{\mbox{$\mathbb Z$}}

\def\OOO{{\mathcal O}}

\def\wt{\widetilde}

\def\cbar{\widehat{\CC}}
\def\smm{{\backslash}}
\def\m{\text{mod }}

\title{Wandering continua for rational maps
}
\author{Guizhen Cui \thanks{supported by NSFC grant no. 11125106} and Yan Gao}
\date{\today}

\begin{document}
\maketitle
\begin{abstract}
We prove that a Latt\`{e}s map admits an always full wandering continuum if and only if it is flexible. The full wandering continuum is a line segment in a bi-infinite or one-side-infinite geodesic under the flat metric.
\end{abstract}

\section{Introduction}

Let $f$ be a rational map of the Riemann sphere $\cbar$ with $\deg f\ge 2$. Denote by $J_f$ and $F_f$ the Julia set and the Fatou set of $f$ respectively. One may refer to \cite{Mi1} for their definitions and basic properties. By a {\bf{continuum}} we mean a connected compact set consisting of more than one point. A continuum $K\subset\cbar$ is called a {\bf wandering continuum} for $f$ if $K\subset J_f$ and $f^n(K)\cap f^m(K)=\emptyset$ for any $n>m\geq 0$.

The existence of wandering continua for polynomials has been studied by many authors. It was proved that all wandering components of the Julia set of a polynomial with disconnected Julia set are points \cite{BH, KS, QY}. For polynomials with connected Julia sets, it was proved that a polynomial without irrational indifferent periodic cycles has no wandering continuum if and only if the Julia set is locally connected \cite{BL, K1, K2, Le, Th}.

The situation for non-polynomial rational maps is different. There are hyperbolic rational maps which have non-degenerate wandering components of their Julia sets. The first example was given by McMullen, where the wandering Julia components are Jordan curves \cite{Mc2}. In fact, it was proved that for a geometrically finite rational map, a wandering component of its Julia set is either a Jordan curve or a single point \cite{PT}.

In this work we study wandering continua for rational maps with connected Julia sets. A continuum $K\subset\cbar$ is called {\bf full} if $\cbar\smm K$ is connected. A wandering continuum $K$ for a rational map $f$ is {\bf always full} if $f^n(K)$ is full for all $n\ge 0$. Refer to \cite{CPT} for the the following theorem and the definition of Cantor multicurves.

\begin{theoremA}
Let $f$ be a post-critically finite rational map and $K\subset J_f$ be a wandering continuum. Then either $K$ is always full or there exists an integer $N\ge 0$ such that $f^n(K)$ is a Jordan curve for $n\ge N$. The latter case happens if and only if $f$ has a Cantor multicurve.
\end{theoremA}

\noindent{\bf Problem}: Under what condition does a post-critically finite rational map $f$ admit an always full wandering continuum?

\vskip 0.24cm

In this paper, we solve this problem for Latt\`{e}s maps (refer to \S2 for its definition). Here is the main theorem:

\begin{theorem}\label{continuum}
A Latt\`{e}s map $f$ admits an always full wandering continuum if and only if it is flexible. In this case the wandering continuum is a line segment in an infinite geodesic under the flat metric.
\end{theorem}

\section{Latt\`{e}s maps}

This section is a review about Latt\`{e}s maps. Refer to \cite{Mc1, Mi1, Mi2} for details. Let $f:\cbar\to\cbar$ be a rational map with $\deg f\ge 2$. Denote by $\deg_z f$ the local degree of $f$ at a point $z\in\cbar$,
$$
\Omega_f=\{z:\, \deg_z f>1\}.
$$
the critical set and
$$
P_f=\overline{\bigcup_{n>0}f^n(\Omega_f)}
$$
the post-critical set of $f$. The rational map $f$ is called {\bf post-critically finite} if $\# P_f<\infty$.

Let $f$ be a post-critically finite rational map. Define $\nu_f(z)$ for each point $z\in\cbar$ to be the least common multiple of the local degrees $\deg_y f^n$ for all $n>0$ and $y\in\cbar$ with $f^n(y)=z$. By convention $\nu_f(z)=\infty$ if the point $z$ is contained in a super-attracting cycle. The {\bf orbiford} of $f$ is defined by $\OOO_f=(\cbar,\nu_f)$. Note that $\nu_f(z)>1$ if and only if $z\in P_f$. The {\bf signature} of the orbifold $\OOO_f$ is the list of the values of $\nu_f$ restricted to $P_f$. The Euler Characteristic of $\OOO_f$ is given by
$$
\chi(\OOO_f)=2-\sum_{z\in\cbar}\left(1-\frac{1}{\nu_f(z)}\right).
$$
It turns out in \cite{Mc1} that $\chi(\OOO_f)\le 0$. The orbifold $\OOO_f$ is {\bf hyperbolic} if $\chi(\OOO_f)<0$, and {\bf parabolic} if $\chi(\OOO_f)=0$.
It is easy to check that the signature of a parabolic orbifold $\OOO_f$ can only be $(\infty, \infty)$, $(2,2,\infty)$, $(2,2,2,2)$, $(3,3,3)$, $(2,4,4)$ or $(2,3,6)$.

Suppose that the signature of $\OOO_f$ is $(\infty, \infty)$. Then $f$ is M\"{o}bius conjugate to a power map $z\mapsto z^d$ with $|d|\ge 2$. Suppose that the signature of $\OOO_f$ is $(2,2,\infty)$. Then  $f$ is M\"{o}bius conjugate to $\pm\Psi_d$, where $\Psi_d$ is the {\bf Chebyshev polynomial} of degree $d$ defined by the equation
$$
\Psi_d(z+\frac 1z)=z^d+\frac{1}{z^d}.
$$
Note that the Julia set of the map $\pm\Psi_d$ is the interval $[-2,2]$. Thus in both cases, there exist no wandering continua for $f$.

\vskip 0.24cm

A post-critically finite rational map $f$ with parabolic orbifold is called a {\bf Latt\`{e}s map} if $\nu_f(z)\neq\infty$ for any point $z\in\cbar$. Let $\nu(\OOO_f)=\max\{\nu_f(z):\, z\in\cbar\}$. Refer to \cite[Theorem 3.1]{Mi2} for the following theorem.

\begin{theorem}\label{lift}
Let $f$ be a Latt\`{e}s map. Then there exist a lattice $\Lambda=\{n+m\omega,\, n,m\in\ZZ\}$ ($\text{Im }\omega>0$), a finite holomorphic cover $\Theta: \CC/\Lambda\to\OOO_f$, a finite cyclic group $G$ of order $\nu(\OOO_f)$ generated by a conformal self-map $\rho$ of $\CC/\Lambda$ with fixed points, and an affine map $A(z)=az+b\,(\m\Lambda): \CC/\Lambda\to\CC/\Lambda$, such that
$$
\Theta(z_1)=\Theta(z_2)\Leftrightarrow z_1=\rho^n(z_2)\text{ for }n\in\ZZ,
$$
and the following diagram commutes:
$$
\begin{array}{ccc}
\CC/\Lambda & \overset{A}{\longrightarrow} & \CC/\Lambda  \\
\Theta \Big\downarrow & & \Big \downarrow \Theta  \\
\OOO_f & \overset{f}{\longrightarrow} & \OOO_f.
\end{array}
$$
\end{theorem}

A Latt\`{e}s map $f$ is called {\bf flexible} if $\OOO_f$ has the signature $(2,2,2,2)$ and the affine map $A(z)=az+b\,(\m\Lambda): \CC/\Lambda\to\CC/\Lambda$ defined in Theorem \ref{lift} has an integer derivative $A'=a\in\ZZ$. A Latt\`{e}s map admits a non-trivial quasiconformal deformation if and only if it is flexible by the following discussion.

Let $f$ be a Latt\`{e}s map. If $\#P_f=3$ and $f$ is topologically conjugate to another rational map $g$, then $f$ and $g$ are M\"{o}bius conjugate.

Now we assume that $\#P_f=4$. Then the signature of $\OOO_f$ is $(2,2,2,2)$ and $\nu(\OOO_f)=2$. Let $\wt\rho:\CC\to\CC$ be a lift of the generator $\rho$ of $G$ under the natural projection $\pi: \CC\to\CC/\Lambda$. Let $z_0\in\CC$ be the unique fixed point of $\wt\rho$. Then $\wt\rho(z)=2z_0-z$. Denote by $Q\subset\CC/\Lambda$ the set of fixed points of $\rho$. Then $\#Q=4$ and $\Theta(Q)=P_f$. Therefore
\begin{equation}\label{1}
\pi^{-1}(Q)=\{n/2+m\omega/2+z_0,\, n,m\in\ZZ\}.
\end{equation}

Let $A(z)=az+b\, (\m\Lambda): \CC/\Lambda\to\CC/\Lambda$ be the affine map defined in Theorem \ref{lift}. Write $\alpha(z)=az+b$. Since $f(P_f)\subset P_f$, we have $\alpha(\pi^{-1}(Q))\subset\pi^{-1}(Q)$. Equivalently, there exist integers $(p,q,r,s)$ such that
\begin{equation}\label{2}
a=p+q\omega,\text{ and } a\omega=r+s\omega.
\end{equation}
It follows that
\begin{equation}\label{3}
q\omega^2+(p-s)\omega-r=0.
\end{equation}

If $a$ is a real number, then $q=r=0$ and $a=p=s$. Thus the real number $a$ must be an integer and equations (2) hold for any complex number $\omega$. This shows that one can make a quasiconformal deformation for the map $f$ to get another rational map such that they are not M\"{o}bius conjugate.

If $a$ is not real, then $q\neq 0$ and thus the complex number $\omega$ with $\text{Im }\omega>0$ is uniquely determined by the integers $(p,q,r, s)$ from equation (3). This shows that if the map $f$ is topologically conjugate to another rational map $g$, then $f$ and $g$ are M\"{o}bius conjugate.

\vskip 0.24cm

{\bf Remark}. A Latt\`{e}s map is flexible if and only if it has a Cantor multicurve. Therefore a Latt\`{e}s map admits a wandering Jordan curve if and only if it is flexible by Theorem A.

\section{Wandering continua for torus coverings}

Let $\Lambda=\{n+m\omega: n,m\in\ZZ\}$ ($\text{Im }\omega>0$) be a lattice. Then $\CC/\Lambda$ is a torus.  A continuum $E\subset\CC/\Lambda$ is {\bf full} if there exists a simply connected domain $U\subset\CC/\Lambda$ such that $E\subset U$ and $U\smm E$ is connected. Let $\pi:\, \CC\to\CC/\Lambda$ be the natural projection. If $E\subset\CC/\Lambda$ is a full continuum, then so is each component of $\pi^{-1}(E)$. In this section, we will prove the following theorem.

\begin{theorem}\label{torus}
Let $A(z)=az+b\,(\m\Lambda):\, \CC/\Lambda\to\CC/\Lambda$ be a covering of the torus with $\deg A\ge 2$. Then the map $A$ admits an always full wandering continuum $E$ if and only if its derivative $a$ is an integer. In this case, the wandering continuum $E$ is a line segment in an infinite geodesic under the flat metric of $\CC/\Lambda$.
\end{theorem}

The proof of Theorem \ref{torus} is based on the following lemmas.

\begin{lemma}\label{topology}
Let $E\subset\CC/\Lambda$ be a full continuum. For any line $L\subset\CC$ and any connected component $B$ of $\pi^{-1}(E)$, if $I$ is a bounded component of $L\smm B$, then $\pi$ is injective on $I$.
\end{lemma}

\begin{proof}
Let $I$ be a bounded component of $L\smm B$. Then there are exactly two components $U,V$ of $\CC\smm(L\cup B)$ such that their boundaries contain the interval $I$. We claim that at least one of them, denoted it by $U$, is bounded. Otherwise one may find a Jordan curve $\gamma$ in $U\cup V\cup I$ such that $\gamma$ separates the two endpoints $x_1$ and $y_1$ of $I$. Since $\gamma$ is disjoint from $B$, and both $x_1$ and $y_1$ are contained in $B$, this contradicts the fact that $B$ is connected.

Assume by contradiction that $\pi$ is not injective on $I$, i.e. there exist two distinct points $x,y\in I$ such that $\pi(x)=\pi(y)$. For each connected component $G$ of $B\cap\partial U$, the set $G\cap L$ is non-empty. Denote by $H(G)$ the closed convex hull of $G\cap L$, i.e. $H(G)$ is the minimal closed interval in $L$ with $H(G)\supset G\cap L$. Then for any two components $G_1, G_2$ of $B\cap\partial U$, $H(G_1)$ and $H(G_2)$ are either disjoint or one contains another. In particular, there exists a component $G_0$ of $B\cap\partial U$ such that $H(G_0)\supset H(G)$ for any component $G$ of $B\cap\partial U$. Moreover, $H(G_0)\supset I$.

\begin{figure}[htbp]\centering
\includegraphics[width=8cm]{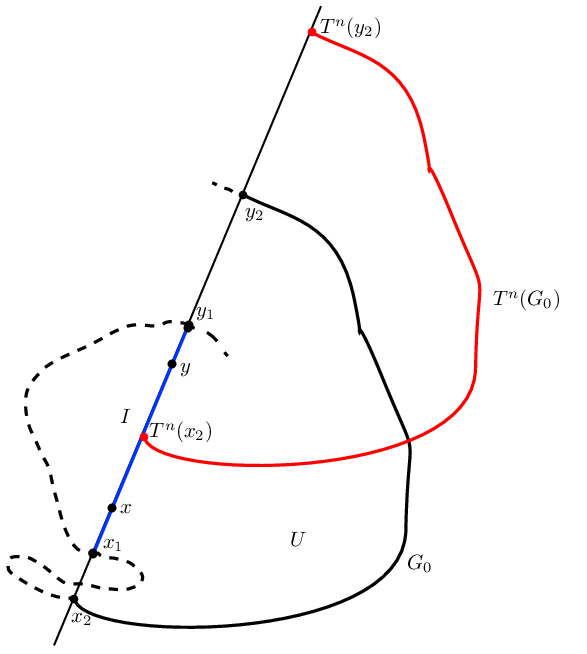}
\begin{center}{\sf Figure 1. Lifting of a full continuum.}
\end{center}
\end{figure}

Set $T(z)=z+(y-x)$. Let $x_2$ and $y_2$ be the two endpoints of $H(G_0)$. Then there exists an integer $n$ such that $T^n(x_2)\in [x,y]$ and hence $T^n(y_2)\notin I$. Let $\HH$ be the component of $\CC\smm L$ that contains $U$, then $T^n(G_0)$ is a continuum in $\HH\cup L$ joining $T^n(x_2)$ with $T^n(y_2)$, whereas $G_0$ is a continuum in $\HH\cup L$ joining $x_2$ with $y_2$. Thus $G_0$ must intersect $T^n(G_0)$.

On the other hand, since $\pi(x)=\pi(y)$, we have $(x-y)\in\Lambda$. Thus $T^n(z)=z\mod\Lambda$ and $T^n(B)$ is another component of $\pi^{-1}(E)$ and hence is disjoint from $B$. This contradicts the facts that $G_0\subset B$ and $G_0$ intersects $T^n(G_0)$.
\end{proof}

\begin{lemma}\label{segment}
Let $A(z)=az+b\,(\m\Lambda): \CC/\Lambda\to\CC/\Lambda$ be a covering with $\deg A\ge 2$ and $E\subset\CC/\Lambda$ be an always full wandering continuum. Then $E$ must be a line segment.
\end{lemma}

\begin{proof} Let $B$ be a component of $\pi^{-1}(E)$. Assume by contradiction that $B$ is not a line segment. We claim that there exists a line $L\subset\CC$ such that $L\smm B$ has a bounded component $I$. Otherwise, each line segment joining two points in $B$ must be contained in $B$. Thus $B$ is convex and hence has positive measure since it is not a line segment. This is impossible since $A$ is expanding and $E$ is wandering.

As in the proof of Lemma \ref{topology}, there exists a bounded component $U$ of $\CC\smm(L\cup B)$ such that $I\subset\partial U$. Since $\deg A\ge 2$, we have $a\neq 1$. Thus the map $\alpha(z)=az+b:\, \CC\to\CC$ has a unique fixed point $z_0\in\CC$. Denote by $\Gamma_0=\{n+m\omega+z_0: n,m\in\ZZ\}$ and $\Gamma_n=\alpha^{-n}(\Gamma_0)$. Then there exists two distinct points $x_n, y_n\in U\cap\Gamma_n$ for some integer $n\ge 0$ such that for the line $L_n$ that passes through the points $x_n, y_n$, the set $L_n\cap U$ has a component $I_n$ which contains both $x_n$ and $y_n$, and the two endpoints of $I_n$ are contained in $B$.

\begin{figure}[htbp]\centering
\includegraphics[width=9cm]{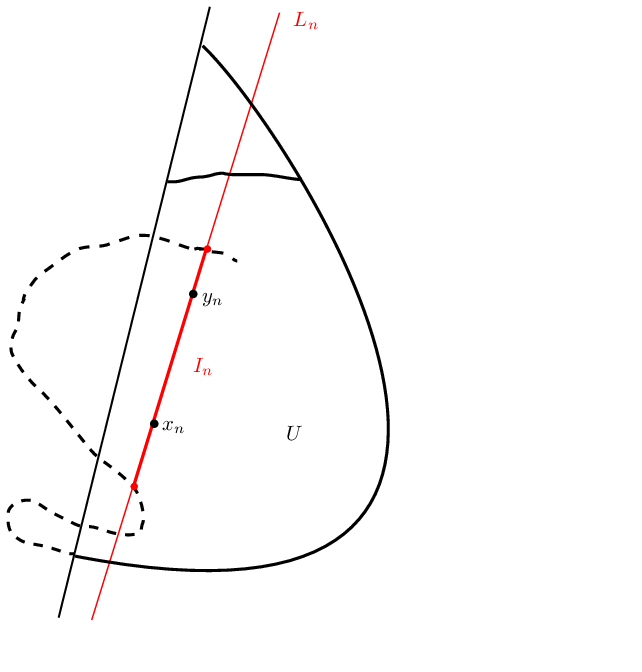}
\begin{center}{\sf Figure 2. A wandering continuum is a line segment. }
\end{center}
\end{figure}

Now consider the full continuum $\alpha^n(B)$ and the line $\alpha^n(L_n)$. The set $\alpha^n(L_n)\smm\alpha^n(B)$ has a component $\alpha^n(I_n)$, which contains $x=\alpha^n(x_n)$ and $y=\alpha^n(y_n)$. Since $x_n,y_n\in\Gamma_n$, we have $x,y\in\Gamma_0$ and hence $\pi(x)=\pi(y)$. This
contradicts Lemma \ref{topology}.
\end{proof}

\begin{lemma}\label{non}
Let $A(z)=az+b\,(\m\Lambda): \CC/\Lambda\to\CC/\Lambda$ be a torus covering with $\deg A\ge 2$. If $a$ is not real, then any line segment in $\CC/\Lambda$ is not wandering.
\end{lemma}

\begin{proof}
Let $E\subset\CC/\Lambda$ be a line segment. We want to show that there exists an integer $n>0$ such that $A^n(E)$ intersects $A^{n+1}(E)$.

Let $R$ be the full parallelogram with vertices $0,1,\omega$ and $1+\omega$. Then $R$ is a fundamental domain of the group $\Lambda$. Thus for any $n\ge 0$, the set $\pi^{-1}(A^n(E))$ has a component $B_n$ such that the midpoint $m(B_n)$ of the line segment $B_n$ is contained in the closure of $R$. Since the diameter of $R$ is less than $1+|\omega|$, for any $n\ge 0$, the Euclidean distance
\begin{equation}\label{4}
|m(B_n)-m(B_{n+1})|\le 1+|\omega|.
\end{equation}

Denote by $a=|a|\exp(i\theta)$. Then $0<|\theta|<\pi$ since $a$ is not real. Let $L_n$ be the line containing $B_n$ for $n\ge 0$. Then $L_n$ and $L_{n+1}$ must intersect at a point $O_n$ and the angle formed by these two lines is $|\theta|$. If $B_n$ is disjoint from $B_{n+1}$, then $O_n\notin B_n$ or $O_n\notin B_{n+1}$. In the former case, we have
$$
|O_n-m(B_n|\ge \frac{|B_n|}{2},
$$
where $|B_n|$ is the length of $B_n$. Therefore the Euclidean distance from $m(B_n)$ to $L_{n+1}$ satisfies
$$
d(m(B_n), L_{n+1})\ge\frac{|B_n|}{2}\sin|\theta|.
$$
It follows that
$$
1+|\omega|\ge |m(B_n)-m(B_{n+1})|\ge d(m(B_n), L_{n+1})\ge\frac{|B_n|}{2}\sin|\theta|.
$$
So
\begin{equation}\label{5}
|B_{n}|\le\frac{2(1+|\omega|)}{|\sin\theta|}.
\end{equation}

\begin{figure}[htbp]\centering
\includegraphics[width=8cm]{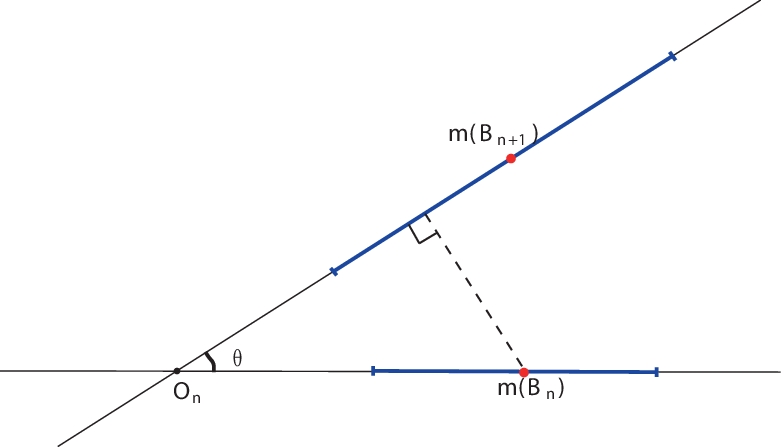}
\begin{center}{\sf Figure 3. The upper bound of the length.}
\end{center}
\end{figure}

In the latter case, we have:
\begin{equation}\label{6}
|B_{n+1}|\le\frac{2(1+|\omega|)}{|\sin\theta|}.
\end{equation}
Noticing that $\deg A=|a|^2\ge 2$, we have $|B_n|=|a|^n|B_0|\to\infty$ as $n\to\infty$. Thus both cases are impossible.
\end{proof}

Now suppose that $A(z)=az+b\,(\m\Lambda): \CC/\Lambda\to\CC/\Lambda$ is a covering with $\deg A\ge 2$ and $a$ is an integer. Let $L\subset\CC$ be a line. Then either $\pi(L)$ is a Jordan curve on $\CC/\Lambda$ or $\pi$ is injective on $L$. Write $\alpha(z)=az+b$. Then for any $n, m\ge 0$, $\alpha^n(L)$ and $\alpha^m(L)$ either coincide or are parallel. Thus if $\pi$ is injective on $L$, then $\pi(L)$ is either eventually periodic or a {\bf wandering line}, i.e. $A^n(\pi(L))\cap A^m(\pi(L))=\emptyset$ for any $n>m\ge 0$.

\begin{lemma}\label{flexible}
Let $L\subset\CC$ be a line and $B\subset L$ be a line segment.

(a) If $\pi(L)$ is a Jordan curve, then $A^n(\pi(B))$ is not full when $n$ is large enough.

(b) If $\pi(L)$ is a wandering line, then $\pi(B)$ is a wandering continuum.

(c) If $\pi(L)$ is an eventually periodic line, then there exists a line segment $B_0\subset B$ such that $\pi(B_0)$ is a wandering continuum.
\end{lemma}

\begin{proof}
(a) Since $\pi(L)$ is a Jordan curve, there exist two distinct points $x, y\in L$ such that $\pi(x)=\pi(y)$. Since $\deg A=|a|^2\ge 2$, there exists an integer $n_0>0$ such that the Euclidean length $|\alpha^n(B)|\ge |x-y|$ when $n\ge n_0$. Thus $A^n(\pi(B))=\pi(\alpha^n(B))=\pi(L)$, which is a  Jordan curve, since $a$ is real.

(b) This is obviously.

(c) Assume that $\pi(L)$ is periodic with period $p\ge 1$ for simplicity. Since $\deg A\ge 2$, there exists a unique point $x_0\in L$ such that $A^p(\pi(x_0))=\pi(x_0)$. Pick a point $y_0$ in the interior of $B$ with $y_0\neq x_0$. Then for any $n\ge 1$, there exists a unique point $y_n\in L$ such that $A^{np}(\pi(y_0))=\pi(y_n)$. Moreover, $y_n\to\infty$ as $n\to\infty$.

Suppose that the integer $a$ is positive. Then all points $y_n$ are contained in the same component of $L\smm\{x_0\}$. Since $y_0$ is contained in the interior of $B$, there exists a closed line segment $B_0\subset B$ such that $B_0\subset (y_0, y_1)$. Then $\pi(B_0)$ is a wandering continuum.

Now suppose that the integer $a$ is negative. Then the points $y_{2k}$ are contained in the same component of $L\smm\{x_0\}$ for $k\ge 0$. Since $y_0$ is contained in the interior of $B$, there exists a closed line segment $B_0\subset B$ such that $B_0\subset (y_0, y_2)$. Then $\pi(B_0)$ is a wandering continuum for $A$.
\end{proof}

\vskip 0.24cm

{\noindent\it Proof of Theorem \ref{torus}}. The result follows directly from Lemmas \ref{topology}, \ref{segment}, \ref{non} and \ref {flexible}. \qed

\section{Proof of Theorem \ref{continuum}}

\noindent{\it Proof of Theorem \ref{continuum}}. Let $f$ be a Latt\`{e}s map. By Theorem \ref{lift}, there exist a lattice $\Lambda=\{n+m\omega,\, n,m\in\ZZ\}$ with $\text{Im }\omega>0$, a finite holomorphic cover $\Theta: \CC/\Lambda\to\OOO_f$, a finite cyclic group $G$ of order $\nu(\OOO_f)$ generated by a conformal self-map $\rho$ of $\CC/\Lambda$ with fixed points, and an affine map $A(z)=az+b\,(\m\Lambda): \CC/\Lambda\to\CC/\Lambda$, such that
$$
\Theta(z_1)=\Theta(z_2)\Leftrightarrow z_1=\rho^n(z_2)\text{ for } n\in\ZZ,
$$
and the following diagram commutes:
$$
\begin{array}{ccc}
\CC/\Lambda & \overset{A}{\longrightarrow} & \CC/\Lambda  \\
\Theta \Big\downarrow & & \Big \downarrow \Theta  \\
\OOO_f & \overset{f}{\longrightarrow} & \OOO_f.
\end{array}
$$

Suppose that $K$ is an always full wandering continuum of the map $f$. Then for each $n\ge 0$, every component of $\Theta^{-1}(f^n(K))$ is a full continuum in $\CC/\Lambda$ since $f^n(K)$ is disjoint from $P_f$. Let $E$ be a component of $\Theta^{-1}(K)$. It is an always full wandering continuum for the map $A$. Therefore the derivative $a$ is an integer and $K$ is a line segment in an infinite geodesic under the flat metric on $\CC/\Lambda$ by Theorem \ref{torus}.

Let $E_n$ be a component of $\Theta^{-1}(f^n(K))$. Then $\rho(E_n)$ is also a component of $\Theta^{-1}(f^n(K))$, where $\rho$ is the generator of the group $G$. Let $R$ be the full parallelogram with vertices $0,1,\omega$ and $1+\omega$. Then $R$ is a fundamental domain of the group $\Lambda$. Thus there are components $I_n$ and $J_n$ of $\pi^{-1}(E_n)$ and $\pi^{-1}(\rho(E_n))$ respectively, such that the midpoints of $I_n$ and $J_n$ are contained in $R$.

Assume that $\#P_f=3$. Let $\wt\rho$ be a lift of the map $\rho$ under the projection $\pi$. Then $\wt\rho$ is a rotation around its fixed point with angle $(2\pi)/\nu$, where $\nu=3,4$ or $6$. Thus the angle formed by the two lines containing $I_n$ and $J_n$ respectively, is $(2\pi)/\nu$. As in the proof of Lemma \ref{non}, we have:
$$
|I_n|\le\frac{2(1+|\omega|)}{\sin(\pi/3)},
$$
where $|I_n|$ is the length of $I_n$. This leads to a contradiction since $|I_n|=|a|^n|I_0|\to\infty$ as $n\to\infty$.
Therefore $\#P_f=4$ and $f$ is a flexible Latt\`{e}s map.

Conversely, suppose that the map $f$ is flexible. Denote by $Q$ the set of fixed points of $\rho$. Then
$A(Q)\subset Q$ since $f(P_f)\subset P_f$.

Let $L\subset\CC$ be a line. If $L$ passes through at least two points $z_1, z_2\in\pi^{-1}(Q)$, then it passes through the point $(z_2+(z_2-z_1))$. Note that $z_2-z_1\equiv 0\,\m(\Lambda/2)$ by (1). So we have $2(z_2-z_1)\equiv 0\,\m\Lambda$, i.e., $\pi(z_2+(z_2-z_1))=\pi(z_1)$. So $\pi(L)$ is a Jordan curve on $\CC/\Lambda$. If $L$ passes through exactly one point $z_0\in\pi^{-1}(Q)$, then $\pi$ is injective on $L$, $\Theta(\pi(L))$ is a ray from a point in $P_f$, and $\Theta:\, \pi(L)\to\Theta(\pi(L))$ is a folding with exactly one fold point at $\pi(z_0)\in Q$. If $L$ is disjoint from $\pi^{-1}(Q)$, then $\Theta\circ\pi$ is injective on $L$.

Suppose that $\pi(L)\subset\CC/\Lambda$ is a wandering line. Then $A^n(\pi(L))$ is disjoint from $Q$ for all $n\ge 0$. Thus $\Theta$ is injective on each line $A^n(\pi(L))$. On the other hand, if $\Theta(A^n(\pi(L)))$ intersects $\Theta(A^{n+p}(\pi(L)))$ for some integer $p>0$, then they coincide since the map $\rho$ in $G$ preserves the slopes of the lines. Thus $A^{n+p}(\pi(L)))=\rho(A^n(\pi(L))$. Therefore
$$
A^{n+2p}(\pi(L))=\rho(A^{n+p}(\pi(L)))
$$
since $A\circ\rho=\rho\circ A$. But $\rho^2$ is the identity. So $A^{n+2p}(\pi(L))=A^{n}(\pi(L))$. This is a contradiction. Therefore $\Theta(A^n(\pi(L)))$ is pairwise disjoint. Thus for any line segment $E\subset\pi(L)$, the set $\Theta(E)$ is an always full wandering continuum for $f$.

Now suppose that $\pi(L)\subset\CC/\Lambda$ is an eventually periodic line with period $p\ge 1$. Then $\Theta(A^n(\pi(L)))$ are either bi-infinite or one-side-infinite geodesics depending on whether $A^n(\pi(L))$ passes through a point in $Q$. Since $\rho^2$ is the identity, either they are disjoint and have the same period, or two of them coincide and the period is $p/2$. Let $E\subset\pi(L)$ be a wandering line segment. In the former case, $\Theta(E)$ is an always full wandering continuum of $f$. In the latter case, there exists a line segment $E_0\subset E$ such that $\Theta(E_0)$ is an always full wandering continuum for $f$. \qed

\newpage
\noindent
Guizhen Cui \\
Academy of Mathematics and Systems Science \\
Chinese Academy of Sciences, Beijing 100190 \\
P. R. China. \\
Email: gzcui@math.ac.cn

\vskip 0.24cm

\noindent
Yan Gao \\
Mathemaitcal School \\
Sichuan University \\
P. R. China \\
Email: gyan@mail.ustc.edu.cn

\end{document}